\documentclass[final]{siamltex}

\usepackage{amssymb}
\usepackage[leqno]{amsmath}
\usepackage{mathrsfs}
\usepackage{stmaryrd}
\usepackage{chemarrow}
\usepackage{algorithm2e}
\usepackage{enumerate}
\usepackage{graphicx}
\usepackage{booktabs, bigstrut, multirow}
\usepackage[pdftex,bookmarksnumbered,bookmarksopen,colorlinks,linkcolor=red,anchorcolor=black,citecolor=blue,urlcolor=blue]{hyperref}
\usepackage[all]{xy}
\usepackage{tikz}
\usetikzlibrary{arrows}

\newtheorem{remark}[theorem]{Remark}

\title{Morley-Wang-Xu element methods with penalty for a fourth order elliptic singular perturbation problem}%
\author{
 Wenqing Wang\thanks{Department of Basic Teaching, Wenzhou Business College, Wenzhou 325035, China (wangwenqing81@hotmail.com).}
\and
Xuehai Huang\thanks{Corresponding author. College of Mathematics and Information Science, Wenzhou University, Wenzhou 325035, China (xuehaihuang@gmail.com). The work of this author was supported by the NSFC Projects 11771338 and 11671304, Zhejiang Provincial
Natural Science Foundation of China Projects LY17A010010, LY15A010015 and LY15A010016, and Wenzhou Science and Technology Plan Project G20160019.}
\and
Kai Tang\thanks{College of Mathematics and Information Science, Wenzhou University, Wenzhou 325035, China (790251532@qq.com, 1134612921@qq.com).}
\and
Ruiyue Zhou\footnotemark[3]
}%

\begin{document}

\maketitle

\begin{abstract}
Two Morley-Wang-Xu element methods with penalty for the fourth order elliptic singular perturbation problem are proposed in this paper,
including the interior penalty Morley-Wang-Xu element method and the super penalty Morley-Wang-Xu element method.
The key idea in designing these two methods is combining the Morley-Wang-Xu element and penalty formulation for the Laplace operator.
Robust a priori error estimates are derived under minimal regularity assumptions on the exact solution by means of some established a posteriori error estimates. Finally, we present some numerical results to demonstrate the theoretical estimates.
\end{abstract}
\begin{keywords}
Fourth Order Singular Perturbation Problem, Morley-Wang-Xu Element, Interior Penalty Method, Super Penalty Method, Error Analysis
\end{keywords}

\section{Introduction}

Assume that $\Omega\subset \mathbb{R}^d$ with $d=2, 3$ is a bounded
polytope. Let $f\in L^2(\Omega)$,
then the fourth order elliptic singular perturbation problem
is to find $u\in H_0^2(\Omega)$ satisfying
\begin{equation}\label{fourthorderpertub}
  \begin{cases}
  \varepsilon^2\Delta^{2}u-\Delta u=f \quad\;\; \textrm{in}~\Omega, \\
  u=\partial_nu=0 \quad\quad\quad \textrm{on}~\partial\Omega,
  \end{cases}
\end{equation}
where $n$ is the unit outward normal to $\partial\Omega$, and $\varepsilon$ is a real small and positive parameter.
Problem~\eqref{fourthorderpertub} models the thin buckling plates with $u$ being the displacement in two dimensions \cite{Frank1997}.
It can also be viewed as a simplification of the stationary Cahn-Hilliard equation in three dimensions \cite{Novick-Cohen2008} and the linearization of the vanishing moment method for the fully nonlinear Monge-Amp$\grave{\textrm{e}}$re equation \cite{BrennerGudiNeilanSung2011}.
The variational formulation of problem~\eqref{fourthorderpertub} is to find $u \in H^{2}_{0} (\Omega)$ such that
\begin{equation}\label{weak form}
  \varepsilon^2 a(u,v)+b(u,v)=(f,v) \quad \forall ~v \in H^{2}_{0}(\Omega),
\end{equation}
where
\begin{gather*}
a(u,v)=(\nabla^{2}u, \nabla^{2}v), \quad b(u,v)=(\nabla u,\nabla v)
\end{gather*}
with $(\cdot,\cdot)$ being the $L^{2}$ inner product over $\Omega$.

The $H^2$-conforming finite elements are suitable to discretize the fourth order operator and the second order operator in problem~\eqref{fourthorderpertub} simultaneously \cite{Semper1992}, such as the Argyris element \cite{ArgyrisFriedScharpf1968} and its three-dimensional counterpart \cite{Zhang2009a}, Hsieh-Clough-Toucher element \cite{Ciarlet1978}. However the $H^2$-conforming finite elements require higher degree polynomials or macroelement techniques
which are completely not necessary since the boundary layers of problem~\eqref{fourthorderpertub} (cf. \cite[section~5]{NilssenTaiWinther2001}).
To this end, nonconforming finite element methods are more popular for the fourth order elliptic singular perturbation problem.
To be specific, many $H^2$-nonconforming elements which are $H^1$-conforming have been constructed for problem~\eqref{fourthorderpertub} in \cite{NilssenTaiWinther2001, TaiWinther2006, GuzmanLeykekhmanNeilan2012, WangWuXie2013, ChenChenQiao2013, ChenChen2014, XieShiLi2010, ChenChenXiao2014, WangShiXu2007, WangShiXu2007a}. And several $H^1$-nonconforming elements for problem~\eqref{fourthorderpertub} were studied in \cite{ChenZhaoShi2005, ChenLiuQiao2010, WangWuXie2013}. In \cite{BrennerNeilan2011}, the $C^0$ interior penalty discontinuous Galerkin (IPDG) method was devised for problem~\eqref{fourthorderpertub}. With the help of some a posteriori error estimates, the authors analyzed the $C^0$ IPDG method under minimal regularity assumptions. And the $C^0$ IPDG method was reanalyzed for a layer-adapted mesh in \cite{FranzRoosWachtel2014}. As a matter of fact, any other discontinuous Galerkin methods for the fourth order elliptic problem  \cite{EngelGarikipatiHughesLarsonEtAl2002, WellsDung2007, HuangHuangHan2010, HuangLai2011, HuangHuang2014, KarakocNeilan2014, AnHuang2015, HuangHuang2016, BabuvskaZlamal1973, MozolevskiSuliBosing2007, SuliMozolevski2007, BrennerGudiSung2010a,GeorgoulisHouston2009} can be used to design robust discretizations for problem~\eqref{fourthorderpertub}.

It is well-known that the Morley-Wang-Xu \cite{Morley1968, WangXu2006} element has the fewest number of degrees of freedom on each element among the existing conforming and nonconforming finite elements for the fourth order problems.
Unfortunately, it has been proved and demonstrated in \cite{NilssenTaiWinther2001, Wang2001} that the Morley element method is divergent for general second order elliptic problems, which indicates that it is also divergent for problem~\eqref{fourthorderpertub} as $\varepsilon\to0$ (cf. \cite[Table~1]{NilssenTaiWinther2001}). To deal with the divergence of the Morley-Wang-Xu element for the Laplace operator,
Wang and his collaborators modified the bilinear formulation corresponding to the Laplace operator by introducing an interpolation operator \cite{WangXuHu2006, WangMeng2007}. The resulting discrete method possessed the sharp half-order convergence rate for the general values of $\varepsilon$, however, which was a lower order convergent finite element method for Poisson equation when $\varepsilon=0$.
The modified Morley-Wang-Xu method in \cite{WangXuHu2006, WangMeng2007} was proved to be robust with respect to the parameter $\varepsilon$.

In this paper, we present two Morley-Wang-Xu element methods with penalty for the fourth order elliptic singular perturbation problem~\eqref{fourthorderpertub}.
In consideration of the simplicity of the Morley-Wang-Xu element, we still use the Morley-Wang-Xu element to discretize problem~\eqref{fourthorderpertub} as
in \cite{WangXuHu2006, WangMeng2007}. Instead of introducing the interpolation operator, the interior penalty discontinuous Galerkin (IPDG) formulation \cite{Arnold1982} is adopted to dispose of the Laplace operator in our first method.
The consistency of the Morley-Wang-Xu element method for the Laplace operator is overcame by the IPDG formulation.
Then in order to simplify the discrete bilinear formulation, we use the super penalty discontinuous Galerkin formulation \cite{BabuvskaZlamal1973, BrezziManziniMariniPietraEtAl2000} to approximate the bilinear formulation $b(\cdot, \cdot)$ in the second method.
The discrete bilinear formulation of the super penalty Morley-Wang-Xu element method has only one more term than the continuous bilinear formulation $b(\cdot, \cdot)$, i.e. the super penalty term
\[
\sum_{F \in \mathcal{F}_{h}} \frac{1}{h_{F}^{2p+1}}(\llbracket u \rrbracket, \llbracket v \rrbracket)_F.
\]
It's worth mentioning that both two Morley-Wang-Xu element methods with penalty in this paper can be generalized to any dimensions.

Borrowing the ideas in \cite{Gudi2010, Gudi2010a, BrennerNeilan2011}, we first establish some local
lower bound estimates of a posteriori error analysis, based on which the a priori error estimates are
derived under minimal regularity assumptions on the exact solution.
To the best of our knowledge, there are few works on the a posteriori error estimates for the fourth order elliptic singular perturbation problem in the literatures. Reliable and efficient a posteriori error estimators were constructed for the nonconforming finite element methods in \cite{ZhangWang2008}.
And some local lower bound estimates of a posteriori error analysis were proved for the $C^0$ IPDG method in \cite{BrennerNeilan2011}.
Recently, robust residual-based a posteriori estimators were developed for Ciarlet-Raviart mixed finite element method in \cite{DuLinZhang2016}.

The influence of the boundary layers of problem~\eqref{fourthorderpertub} is also considered in the a priori error estimates for the interior penalty Morley-Wang-Xu element method in this paper. Both the Morley-Wang-Xu element methods with penalty converge uniformly with respect to the parameter $\varepsilon$. In consideration of the boundary layers, the robust and sharp half-order convergence rate of the interior penalty Morley-Wang-Xu element method is achieved for the general values of $\varepsilon$.
And we improve the convergence rate in the following two cases: (1) the parameter $\varepsilon$ is bounded both from above and below by positive constants independent of the mesh size $h$; (2) $\varepsilon\lesssim h^{\gamma}$ with $\gamma>1$.

The rest of this paper is organized as follows. In Section 2, we provide some notations and the description of the Morley-Wang-Xu element.
The interior penalty Morley-Wang-Xu element method and the super penalty Morley-Wang-Xu element method are presented and analyzed in Section 3 and Section 4 respectively.
Some numerical results are provided in Section 5.

\section{Preliminaries}

Given a bounded domain $G\subset\mathbb{R}^{d}$ and a
non-negative integer $m$, let $H^m(G)$ be the usual Sobolev space of functions
on $G$. The corresponding norm and semi-norm are denoted respectively by
$\Vert\cdot\Vert_{m,G}$ and $|\cdot|_{m,G}$.   Let $(\cdot, \cdot)_G$ be the standard inner product on $L^2(G)$, $(L^2(G))^d$ or $(L^2(G))^{d\times d}$.
If $G$ is $\Omega$, we abbreviate
$\Vert\cdot\Vert_{m,G}$, $|\cdot|_{m,G}$ and $(\cdot, \cdot)_G$ by $\Vert\cdot\Vert_{m}$, $|\cdot|_{m}$ and $(\cdot, \cdot)$,
respectively. Let $H_0^m(G)$ be the closure of $C_{0}^{\infty}(G)$ with
respect to the norm $\Vert\cdot\Vert_{m,G}$.
Let $P_m(G)$ stand for the set of all
polynomials in $G$ with the total degree no more than $m$. Denote by $P_G^m$ the $L^2$-orthogonal projection
operator onto $P_m(G)$. And the range of $P_G^m$ is defined to be zero when $m$ is a negative integer.
For any finite set $\mathcal{S}$, denote by $\#\mathcal{S}$ the cardinality of $\mathcal{S}$.

Suppose the domain $\Omega$ is subdivided by a family of shape regular simplicial grids $\mathcal {T}_h$  (cf.\ \cite{BrennerScott2008,Ciarlet1978}) with $h:=\max\limits_{K\in \mathcal {T}_h}h_K$
and $h_K:=\mbox{diam}(K)$.
Let $\mathcal{F}_h$ be the union of all $d-1$ dimensional faces
of $\mathcal {T}_h$ and $\mathcal{F}^i_h$ the union of all
interior $d-1$ dimensional faces of the triangulation $\mathcal {T}_h$. For any $F\in\mathcal{F}_h$,
denote by $h_F$ its diameter and fix a unit normal vector $n_F$.
For each $K\in\mathcal{T}_h$, denote by $n_K$ the
unit outward normal to $\partial K$. Without causing any confusion, we will abbreviate $n_K$ as $n$ for simplicity.
Set for any face $F\in\mathcal{F}_h$
\[
\partial^{-1}F :=\{K\in\mathcal{T}_h: F\subset \partial K\}, \quad  \omega_F:=\textrm{interior}\left(\bigcup_{K\in\partial^{-1}F}\bar{K}\right).
\]
Discrete differential operators $\nabla_h$ is defined as
the elementwise counterpart of $\nabla$ with respect to $\mathcal{T}_h$.
Throughout this paper, we also use
``$\lesssim\cdots $" to mean that ``$\leq C\cdots$", where
$C$ is a generic positive constant independent of $h$ and the
parameter $\varepsilon$,
which may take different values at different appearances.

Moreover, we introduce averages and jumps on $d-1$ dimensional faces as in \cite{HuangHuangHan2010}.
Consider two adjacent simplices $K^+$ and $K^-$ sharing an interior face $F$.
Denote by $n^+$ and $n^-$ the unit outward normals
to the common face $F$ of the simplices $K^+$ and $K^-$, respectively.
For a scalar-valued function $v$, write $v^+:=v|_{K^+}$ and $v^-
:=v|_{K^-}$.   Then define the average and jump on $F$ as
follows:
\[
\{v\}:=\frac{1}{2}(v^++v^-),
  \quad  \llbracket v\rrbracket:=v^+n_F\cdot n^++v^-n_F\cdot n^-.
\]
On a face $F$ lying on the boundary $\partial\Omega$, the above terms are
defined by
\[
\{v\}:=v, \quad \llbracket v\rrbracket
   :=vn_F\cdot n.
\]

Define two sets by
\[
\mathcal{C}_1:=\{1,2,\cdots,d+1\},\quad \mathcal{C}_2:=\{\boldsymbol{I}=(i_1,i_2);\;1\le i_1<i_2\le d+1\}.
\]
Then for a $d$-simplex $K$ with $d+1$ vertices $p_i$ $(1\leq i\leq d+1)$, and for $j\in \mathcal{C}_1$ and $\boldsymbol{I}=(i_1,i_2)\in \mathcal{C}_2$, denote by $F_j^{K}$ the $(d-1)$-subsimplex of $K$ without $p_j$ as its vertex, while denote by $F_{\boldsymbol{I}}^{K}$ the $(d-2)$-subsimplex of $K$ without $p_{i_1}$ and $p_{i_2}$ as its vertices. 
As usual, $|F|$ denotes  the measure of a given open set $F$. With the partition $\mathcal{T}_h$, the Morley-Wang-Xu element \cite{WangXu2006, Morley1968, WangXu2013} is defined as follows: for any $K\in\mathcal{T}_h$, the local shape function space $P_{K}$ is $P_2(K)$ and the nodal variables are given by
\begin{equation}\label{eq:mwxdofs}
\mathcal{N}_{K}:=\bigg\{\frac{1}{|F_{\boldsymbol{I}}^{K}|}\int_{F_{\boldsymbol{I}}^{K}}v~ds\quad \forall\,\boldsymbol{I}\in \mathcal{C}_2,
\quad \frac{1}{|F_j^{K}|}\int_{F_j^{K}}\partial_{n}v~ds\quad \forall\, j\in \mathcal{C}_1
\bigg\}.
\end{equation}
Associated with the partition $\mathcal{T}_h$, the global Morley-Wang-Xu element space $V_h$ consists of all piecewise quadratic functions
on $\mathcal{T}_h$ such that, their integral average over each $(d-2)$-dimensional face
of elements in $\mathcal{T}_h$ are continuous, and their normal derivatives are continuous at
the barycentric point of each $(d-1)$-dimensional face of elements in $\mathcal{T}_h$.
And
define
\[
V_{h0} =\left\{v\in V_{h}: \frac{1}{|F_{\boldsymbol{I}}^{K}|}\int_{F_{\boldsymbol{I}}^{K}}v~ds
=\frac{1}{|F_j^{K}|}\int_{F_j^{K}}\partial_{n}v~ds=0\quad \forall
\,F_{\boldsymbol{I}}^{K},\,F_j^{K}\subset\partial\Omega\right\}.
\]

\section{Interior penalty Morley-Wang-Xu element method}

In this section,  we will devise an interior penalty Morley-Wang-Xu element method for problem~\eqref{fourthorderpertub}
by using the interior penalty discontinuous Galerkin formulation to approximate the bilinear formulation corresponding to the Laplace operator.

Set discrete bilinear forms
\begin{align*}
  a_{h}(w,v):= & \sum_{K \in \mathcal {T}_h}(\nabla^{2}w, \nabla^{2}v)_K \quad \forall~w, v \in H^{2}(\Omega)+V_{h}, \\
  b_{h}(w,v):= & \sum_{K \in \mathcal {T}_h}(\nabla w, \nabla v)_K - \sum_{F \in \mathcal{F}_{h}}(\{\partial_{n_F} w\} , \llbracket v \rrbracket)_F  - \sum_{F \in \mathcal{F}_{h}}(\{\partial_{n_F} v\} , \llbracket w \rrbracket)_F \\
  & + \sum_{F \in \mathcal{F}_{h}} \frac{\sigma}{h_{F}}(\llbracket w \rrbracket, \llbracket v \rrbracket)_F \quad \forall~ w, v\in H^{1}(\Omega)+V_{h},
\end{align*}
where $\sigma$ is a positive real number.
Then the interior penalty Morley-Wang-Xu (IPMWX) element method for problem~\eqref{fourthorderpertub}
is to find $u_h \in V_{h0} $ such that
\begin{equation}\label{ipdgmwx}
  \varepsilon^2 a_h(u_h,v_h)+b_h(u_h,v_h)=(f,v_h) \quad \forall~v_h \in V_{h0}.
\end{equation}

\begin{remark}\rm
Apart form the interior penalty discontinuous Galerkin (IPDG) formulation \cite{BrezziManziniMariniPietraEtAl2000, BrezziManziniMariniPietraEtAl1999, ArnoldBrezziCockburnMarini2001, PerairePersson2008}, many other discontinuous Galerkin formulations \cite{CockburnKarniadakisShu2000}
can be also used to tackle the divergence of the Morley element for the Laplace operator, such as local discontinuous Galerkin method~\cite{CockburnShu1998}, compact discontinuous Galerkin method~\cite{PerairePersson2008}, weakly over-penalized symmetric interior penalty method~\cite{BrennerOwensSung2012}, etc. 
Among all of these discontinuous Galerkin methods, the IPDG method is one of the simplest formulations to discretize the Laplace operator, and it possesses the compact stencil in the sense that only
the degrees of freedom belonging to neighboring elements are connected.
\end{remark}

For any $w\in H^1(\Omega)+V_h$, define some discrete norms
\begin{gather*}
 |w|_{1,h}^2:=\sum_{K \in \mathcal {T}_h} |w|_{1, K}^2 , \quad |w|_{2,h}^2:=\sum_{K \in \mathcal {T}_h} |w|_{2, K}^2 ,\\
  \interleave w \interleave^2:=|w|_{1,h}^2+\sum_{F \in \mathcal{F}_{h}} h_F^{-1} \| \llbracket w \rrbracket \|_{0,F}^2 ,\quad
  \| w \|_{\varepsilon ,h}^2:=\varepsilon^2 |w|_{2,h}^2+\interleave w \interleave^2 .
\end{gather*}
It follows from \cite{Arnold1982}
\[
b_h(w_h, v_h)\lesssim \interleave w_h \interleave\interleave v_h \interleave, \quad \interleave w_h \interleave^2\lesssim b_h(w_h, w_h)\quad\forall~w_h, v_h\in V_h,
\]
when $\sigma$ is large enough.
Hence we have
\[
\varepsilon^2a_h(w_h, v_h)+b_h(w_h, v_h)\lesssim  \| w_h \|_{\varepsilon ,h}\| v_h\|_{\varepsilon ,h} \quad\forall~w_h, v_h\in V_h,
\]
 \begin{equation}\label{eq:coercivity}
 \| w_h \|_{\varepsilon ,h}^2\lesssim \varepsilon^2 a_h(w_h,w_h)+b_h(w_h,w_h)\ \ \forall~w_h \in V_{h}.
\end{equation}
The wellposedness of the IPMWX method~\eqref{ipdgmwx} follows immediately from \eqref{eq:coercivity}.

To derive the error estimate under minimal regularity of the exact solution $u$,
we first recall some a posteriori error estimates \cite{BrennerNeilan2011}.
\begin{lemma}
Let $u\in H_0^2(\Omega)$ be the solution of problem~\eqref{fourthorderpertub}. Then we have for any $w_h\in V_{h0}$, $K \in \mathcal {T}_h$, and $F \in \mathcal{F}_{h}^{i}$
\begin{align}
\|f+\Delta w_h\|_{0, K} \lesssim &\varepsilon^2 h_K^{-2}|u-w_h|_{2, K}+h_{K}^{-1}|u-w_h|_{1, K} + \|f-Q_K^rf\|_{0, K},
\label{eq:posterror1}\\
\| \llbracket \partial_{nn}^2 w_h \rrbracket \|_{0, F}^{2}\lesssim &\sum_{K \in \partial^{-1}F} \Big( h^{-1}_K |u-w_h|^{2}_{2, K}+ \varepsilon^{-4} h_K|u-w_h|^{2}_{1,K} \notag\\
&\quad\quad\quad\quad\quad+\varepsilon^{-4} h_{K}^{3}\|f-Q_K^rf\|_{0, K}^2 \Big),
\label{eq:posterror2}\\
\|\llbracket \partial_{n_F} w_h \rrbracket \|_{0, F}^{2}\lesssim & \sum_{K \in \partial^{-1}F} \Big(\varepsilon^4 h_{K}^{-3}|u-w_h|_{2, K}^2+h_{K}^{-1}|u-w_h|_{1, K}^2 \notag\\
&\quad\quad\quad\quad\quad+h_{K}\|f-Q_K^rf\|_{0, K}^2 \Big), \label{eq:posterror3}
\end{align}
where $Q_K^r$ is defined as the $L^2$-orthogonal projection operator onto $P_r(K)$ with any non-negative integer r.
\end{lemma}
\begin{proof}
Estimates \eqref{eq:posterror1}-\eqref{eq:posterror3} has been proved in \cite{BrennerNeilan2011} for $d=2$. Here we will focus on these a posteriori error estimates for $d=3$. We can readily derive \eqref{eq:posterror1} in three dimensions as in \cite{BrennerNeilan2011} by standard bubble function argument.
Next we will prove \eqref{eq:posterror2}-\eqref{eq:posterror3} by using bubbles functions in \cite{AnHuang2015} rather than those in \cite{BrennerNeilan2011}.

Let $K_1$ and $K_2$ be two elements in $\mathcal{T}_h$ sharing the common face $F$. With $F$, we associate a face bubble function given by (cf. \cite{Verfurth1996, HuangHuangHan2010})
\[
b_F:=\left\{
\begin{array}{ll}
\lambda_{K_1,1}\lambda_{K_1,2}\lambda_{K_1,3}\lambda_{K_2,1}\lambda_{K_2,2}\lambda_{K_2,3}& \textrm{  in  } \omega_F, \\
0 &\textrm{  in  } \Omega\backslash \omega_F,
\end{array}
\right.
\]
where $\lambda_{K_1,i}$ and $\lambda_{K_2,i}$ for $i=1, 2, 3$ are barycentric coordinates of $K_1$ and $K_2$ associated with three vertices of $F$, respectively.
Suppose the plane face $F$ lying in is governed by $n_F\cdot x+C_1=0$ where $C_1$ is a constant.
By direct manipulation, we readily get
\begin{equation}\label{eq:temp170703-1}
\left|n_F\cdot x+C_1\right|\lesssim h_F \quad \textrm{ in } \omega_F.
\end{equation}
Set $J_{1,F}:=\llbracket \partial_{nn}^2 w_h \rrbracket|_F$. And $\widetilde{J}_{1,F}$ is defined by extending the jump $J_{1,F}$ to $\omega_F$  constantly along the normal to $F$. By the inverse inequality, it holds for each $K\in\partial^{-1}F$
\begin{equation}\label{eq:temp170703-2}
\|\widetilde{J}_{1,F}\|_{0,K}\lesssim h_K^{3/2}\|\widetilde{J}_{1,F}\|_{L^{\infty}(K)}=h_K^{3/2}\|J_{1,F}\|_{L^{\infty}(F)}\lesssim h_F^{1/2}\|J_{1,F}\|_{0,F}.
\end{equation}
Take $\phi_F:=(n_F\cdot x+C_1)b_F^2\widetilde{J}_{1,F}\in H_0^2(\omega_F)$. Then we get from \eqref{eq:temp170703-1}-\eqref{eq:temp170703-2}
\begin{equation}\label{eq:phiF1}
\phi_F=0, \quad \partial_{n_F}\phi_F=b_F^2\widetilde{J}_{1,F}\partial_{n_F}(n_F\cdot x+C_1)=b_F^2\widetilde{J}_{1,F}=b_F^2J_{1,F}
\quad\textrm{ on } F,
\end{equation}
\begin{equation}\label{eq:phiF2}
\sum_{K\in\partial^{-1}F}\|\phi_F\|_{0,K}\lesssim h_F\sum_{K\in\partial^{-1}F}\|\widetilde{J}_{1,F}\|_{0,K}\lesssim h_F^{3/2}\|J_{1,F}\|_{0,F}.
\end{equation}
According to standard scaling argument, \eqref{eq:phiF1} and the fact that $\phi_F=0$ on each edge of $K_1$ and $K_2$,
\begin{align*}
\|J_{1,F}\|_{0,F}^2\lesssim & (J_{1,F}, \partial_{n_F}\phi_F)_F = (\llbracket \partial_{nn}^2 w_h \rrbracket, \partial_{n_F}\phi_F)_F \\
=&\sum_{K\in\partial^{-1}F}(\partial_{nn}^2 w_h, \partial_{n}\phi_F)_{\partial K}=
\sum_{K\in\partial^{-1}F}((\nabla^2w_h)n, \nabla\phi_F)_{0, \partial K}.
\end{align*}
Applying integration by parts, \eqref{eq:phiF1} and \eqref{weak form}, it follows that
\begin{align*}
\varepsilon^2\|J_{1,F}\|_{0,F}^2\lesssim & \varepsilon^2\sum_{K\in\partial^{-1}F}(\nabla^2w_h, \nabla^2\phi_F)_{0, K} \\
=&\varepsilon^2\sum_{K\in\partial^{-1}F}(\nabla^2(w_h-u), \nabla^2\phi_F)_{0, K} + \sum_{K\in\partial^{-1}F}(\nabla(w_h-u), \nabla\phi_F)_{0, K} \\
& + \sum_{K\in\partial^{-1}F}(f+\Delta w_h, \phi_F)_{0, K}.
\end{align*}
Then from Cauchy-Schwarz inequality, \eqref{eq:posterror1} and inverse inequality, we get
\begin{align*}
\varepsilon^2\|J_{1,F}\|_{0,F}^2\lesssim & \sum_{K\in\partial^{-1}F}\left(\varepsilon^2|u-w_h|_{2,K}|\phi_F|_{2,K} + |u-w_h|_{1,K}|\phi_F|_{1,K}\right) \\
&+\sum_{K\in\partial^{-1}F}\|f+\Delta w_h\|_{0,K}\|\phi_F\|_{0,K} \\
\lesssim & \sum_{K\in\partial^{-1}F}\left(\varepsilon^2 h_K^{-2}|u-w_h|_{2, K}+h_{K}^{-1}|u-w_h|_{1, K}\right)\|\phi_F\|_{0,K} \\
& +\sum_{K\in\partial^{-1}F}\|f-Q_K^rf\|_{0,K}\|\phi_F\|_{0,K} ,
\end{align*}
which together with \eqref{eq:phiF2} implies \eqref{eq:posterror2}.

Let $J_{2,F}:=\llbracket \partial_{n_F} w_h \rrbracket|_F$. And $\widetilde{J}_{2,F}$ is defined by extending the jump $J_{2,F}$ to $\omega_F$  constantly along the normal to $F$. Set $\psi_F:=b_F^2\widetilde{J}_{2,F}\in H_0^2(\omega_F)$.  It is easy to check that
\begin{equation}\label{eq:psiF}
\|\psi_F\|_{0,\omega_F}\lesssim h_F^{1/2}\|J_{2,F}\|_{0,F}.
\end{equation}
By standard scaling argument, integration by parts and \eqref{weak form}, it follows that
\begin{align}
\|J_{2,F}\|_{0,F}^2\lesssim & (\llbracket \partial_{n_F} w_h \rrbracket, \psi_F)_{0,F}
=\sum_{K\in\partial^{-1}F}(\partial_{n} w_h,\psi_F)_{\partial K} \notag\\
=& \sum_{K\in\partial^{-1}F}(\nabla w_h, \nabla\psi_F)_{K} + \sum_{K\in\partial^{-1}F}(\Delta w_h,\psi_F)_{K} \notag\\
=& \sum_{K\in\partial^{-1}F}(\nabla (w_h-u), \nabla\psi_F)_{K} + \sum_{K\in\partial^{-1}F}(f+\Delta w_h,\psi_F)_{K} \notag\\
&+ \varepsilon^2\sum_{K\in\partial^{-1}F}(\nabla^2 (w_h-u), \nabla^2\psi_F)_{K} - \varepsilon^2\sum_{K\in\partial^{-1}F}(\nabla^2 w_h, \nabla^2\psi_F)_{K}.\label{eq:temp1}
\end{align}
Noting that $\psi_F=0$ on each edge of $K$ again, we obtain
\[
(\nabla^2 w_h, \nabla^2\psi_F)_{K}=((\nabla^2 w_h)n, \nabla\psi_F)_{\partial K}=(\partial_{nn}^2 w_h, \partial_n\psi_F)_{\partial K}.
\]
Thus
\[
\sum_{K\in\partial^{-1}F}(\nabla^2 w_h, \nabla^2\psi_F)_{K}=(\llbracket\partial_{nn}^2 w_h\rrbracket, \partial_{n_F}\psi_F)_{F}.
\]
Therefore \eqref{eq:posterror3} will be achieved from \eqref{eq:temp1}, \eqref{eq:psiF} and \eqref{eq:posterror1}-\eqref{eq:posterror2}.
\end{proof}

We also need the enriching operator $E_h$ from $V_{h0}$ to some conforming finite element space $V_{h0}^C\subset H_0^2(\Omega)$.
Here we adopt Argyris-Zhang element space \cite{ArgyrisFriedScharpf1968, Zhang2009a} for $V_{h0}^C$ ($5$th Argyris-Zhang element for $d=2$ and $10$th Argyris-Zhang element for $d=3$).
For each $w_h\in V_{h0}$, define $E_hw_h\in V_{h0}^C$ such that for any degree of freedom $N$ located in the interior of $\Omega$,
\[
N(E_hw_h)=\frac{1}{\#\mathcal{T}_N}\sum_{K\in\mathcal{T}_N}N(w_h|_K),
\]
where $\mathcal{T}_N\subset\mathcal{T}_h$ denotes the set of simplices sharing the degree of freedom $N$.
By technique used in \cite{BrennerNeilan2011} and the weak continuity of Morley-Wang-Xu element \cite{WangXu2006}, we have for any $w_h, v_h \in V_{h0}$ and $K\in\mathcal{T}_h$
\begin{equation}\label{eq:Ehprop1}
((\nabla^2w_h)n, \nabla(v_h-E_hv_h))_{0, \partial K}=(\partial_{nn}^2w_h, \partial_n(v_h-E_hv_h))_{0, \partial K},
\end{equation}
\begin{align}
\sum_{K \in \mathcal{T}_{h}} h^{-4}_{K} \|w_h-E_hw_h\|^{2}_{0, K}+ \sum_{K \in \mathcal{T}_{h}} h^{-3}_{K} \|w_h-E_hw_h\|^{2}_{0, \partial K} & \notag\\
+ \sum_{K \in \mathcal{T}_{h}} h^{-1}_{K} \|\partial_{n}(w_h-E_hw_h)\|^{2}_{0, \partial K} & \lesssim |w_h|_{2,h}^{2}, \label{eq:Ehestimate1}
\end{align}
\begin{align}
\sum_{K \in \mathcal{T}_{h}} h^{-2}_{K} \|w_h-E_hw_h\|^{2}_{0, K}+ \sum_{K \in \mathcal{T}_{h}} h^{-1}_{K} \|w_h-E_hw_h\|^{2}_{0, \partial K} & \notag\\
+ \sum_{K \in \mathcal{T}_{h}} h_{K} \|\partial_{n}(w_h-E_hw_h)\|^{2}_{0,\partial K} & \lesssim  |w_h|_{1,h}^2,\label{eq:Ehestimate2}
\end{align}
\begin{gather}
 \|E_{h}w_h\|_{\varepsilon,h} \lesssim \|w_h\|_{\varepsilon,h}. \label{eq:Ehestimate3}
\end{gather}

\begin{lemma}
Let $u\in H_0^2(\Omega)$ be the solution of problem~\eqref{fourthorderpertub}. It holds for any $w_h, v_h\in V_{h0}$
\begin{align}
&(f, v_h-E_hv_h)-\varepsilon^2 a_h(w_h, v_h-E_hv_h) - b_h(w_h, v_h-E_hv_h) \notag\\
\lesssim &\left(\|u-w_h\|_{\varepsilon, h} + \mathrm{osc}_h(f)\right)\|v_h\|_{\varepsilon, h}, \label{eq:nonconformestimate}
\end{align}
where $\mathrm{osc}_h^2(f):=\sum\limits_{K\in\mathcal{T}_h}h_K^2\|f-Q_K^rf\|_{0, K}^2$.
\end{lemma}
\begin{proof}
Applying integration by parts, we get from \eqref{eq:Ehprop1} and the weak continuity of $\partial_{n_F}v_h$
\begin{align*}
a_h(w_h, v_h-E_hv_h)=&\sum_{K\in\mathcal{T}_h}((\nabla^2w_h)n, \nabla(v_h-E_hv_h))_{\partial K} \\
=&\sum_{K\in\mathcal{T}_h}(\partial_{nn}^2w_h, \partial_n(v_h-E_hv_h))_{\partial K}\\
=&\sum_{F\in\mathcal{F}_h^i}(\llbracket\partial_{nn}^2w_h\rrbracket, \{\partial_{n_F}(v_h-E_hv_h)\})_F.
\end{align*}
Using Cauchy-Schwarz inequality and \eqref{eq:posterror2}, it holds
\begin{align*}
&-\varepsilon^2\left(\llbracket\partial_{nn}^2w_h\rrbracket, \{\partial_{n_F}(v_h-E_hv_h)\}\right)_F\leq \varepsilon^2\|\llbracket\partial_{nn}^2w_h\rrbracket\|_{0,F}\|\{\partial_{n_F}(v_h-E_hv_h)\}\|_{0,F} \\
\lesssim & \sum_{K \in \partial^{-1}F} \varepsilon^2  |u-w_h|_{2, K}h^{-1/2}_K\|\{\partial_{n_F}(v_h-E_hv_h)\}\|_{0,F}\\
 & +\sum_{K \in \partial^{-1}F} \Big(|u-w_h|_{1,K}+h_{K}\|f-Q_K^rf\|_{0, K} \Big)h_{K}^{1/2}\|\{\partial_{n_F}(v_h-E_hv_h)\}\|_{0,F}.
\end{align*}
Then we get from \eqref{eq:Ehestimate1}-\eqref{eq:Ehestimate2}
\begin{align}
-\varepsilon^2 a_h(w_h, v_h-E_hv_h)\lesssim & \varepsilon^2 |u-w_h|_{2, h}\|v_h\|_{2,h} + \left(|u-w_h|_{1,h} + \mathrm{osc}_h(f)\right)|v_h|_{1,h} \notag\\
\lesssim & \left(\|u-w_h\|_{\varepsilon, h} + \mathrm{osc}_h(f)\right)\|v_h\|_{\varepsilon, h}. \label{eq:ahtemp}
\end{align}
On the other hand, it follows from integration by parts
\begin{align}
&(f, v_h-E_hv_h) - b_h(w_h, v_h-E_hv_h) \notag\\
=&\sum_{K\in\mathcal{T}_h}(f+\Delta w_h, v_h-E_hv_h)_{K} - \sum_{F \in \mathcal{F}_{h}^i}(\llbracket\partial_{n_F} w_h\rrbracket , \{ v_h-E_hv_h \})_F \notag\\
  & - \sum_{F \in \mathcal{F}_{h}}(\{\partial_{n_F} (v_h-E_hv_h)\} , \llbracket u-w_h \rrbracket)_F + \sum_{F \in \mathcal{F}_{h}} \frac{\sigma}{h_{F}}(\llbracket u-w_h \rrbracket, \llbracket v_h-E_hv_h \rrbracket)_F. \label{eq:temp2}
\end{align}
By \eqref{eq:posterror1}, it holds
\begin{align*}
(f+\Delta w_h, v_h-E_hv_h)_{K} \leq & \|f+\Delta w_h\|_{0,K}\|v_h-E_hv_h\|_{0,K} \\
\lesssim & \varepsilon^2 |u-w_h|_{2, K}h_K^{-2}\|v_h-E_hv_h\|_{0,K} \\
&+ \left(|u-w_h|_{1, K} + h_K\|f-Q_K^rf\|_{0, K}\right)h_{K}^{-1}\|v_h-E_hv_h\|_{0,K}.
\end{align*}
Hence we obtain from \eqref{eq:Ehestimate1}-\eqref{eq:Ehestimate2}
\begin{align}
\sum_{K\in\mathcal{T}_h}(f+\Delta w_h, v_h-E_hv_h)_{K}
\lesssim & \varepsilon^2 |u-w_h|_{2, h}\|v_h\|_{2,h} + \left(|u-w_h|_{1,h} + \mathrm{osc}_h(f)\right)|v_h|_{1,h} \notag\\
\lesssim & \left(\|u-w_h\|_{\varepsilon, h} + \mathrm{osc}_h(f)\right)\|v_h\|_{\varepsilon, h}. \label{eq:fdeltawhtemp}
\end{align}
Due to \eqref{eq:posterror3}, it follows
\begin{align*}
&\|\llbracket\partial_{n_F} w_h\rrbracket\|_{0,F} \|\{ v_h-E_hv_h \}\|_{0,F} \\
\lesssim &\sum_{K \in \partial^{-1}F}\varepsilon^2 |u-w_h|_{2, K}h_{K}^{-3/2}\|\{ v_h-E_hv_h \}\|_{0,F} \\
&+ \sum_{K \in \partial^{-1}F} \left(|u-w_h|_{1, K}+h_{K}\|f-Q_K^rf\|_{0, K}\right)h_{K}^{-1/2}\|\{ v_h-E_hv_h \}\|_{0,F}.
\end{align*}
Combined with Cauchy-Schwarz inequality and \eqref{eq:Ehestimate1}-\eqref{eq:Ehestimate2}, we achieve
\begin{align}
&- \sum_{F \in \mathcal{F}_{h}^i}(\llbracket\partial_{n_F} w_h\rrbracket , \{ v_h-E_hv_h \})_F \notag\\
\leq &\sum_{F \in \mathcal{F}_{h}^i}\|\llbracket\partial_{n_F} w_h\rrbracket\|_{0,F} \|\{ v_h-E_hv_h \}\|_{0,F} \notag\\
\lesssim & \varepsilon^2 |u-w_h|_{2, h}\|v_h\|_{2,h} + \left(|u-w_h|_{1,h} + \mathrm{osc}_h(f)\right)|v_h|_{1,h} \notag\\
\lesssim & \left(\|u-w_h\|_{\varepsilon, h} + \mathrm{osc}_h(f)\right)\|v_h\|_{\varepsilon, h}. \label{eq:partialnwhtemp}
\end{align}
Making use of Cauchy-Schwarz inequality and \eqref{eq:Ehestimate2} again, we get
\begin{align*}
&\sum_{F \in \mathcal{F}_{h}}(\{\partial_{n_F} (v_h-E_hv_h)\}, \llbracket u-w_h \rrbracket)_F - \sum_{F \in \mathcal{F}_{h}} \frac{\sigma}{h_{F}}(\llbracket u-w_h \rrbracket, \llbracket v_h-E_hv_h \rrbracket)_F \\
\lesssim &\sum_{F \in \mathcal{F}_{h}}\left(h^{\frac{1}{2}}_F\|\{\partial_{n_F} (v_h-E_hv_h)\}\|_{0, F}+\frac{1}{h^{\frac{1}{2}}_F}\|\llbracket v_h-E_hv_h \rrbracket\|_{0,F}\right)\frac{1}{h^{\frac{1}{2}}_F}\|\llbracket u-w_h \rrbracket\|_{0,F} \\
\lesssim & \|u-w_h\|_{\varepsilon, h}|v_h|_{1,h} \leq \|u-w_h\|_{\varepsilon, h}\|v_h\|_{\varepsilon,h},
\end{align*}
which together with \eqref{eq:temp2}-\eqref{eq:partialnwhtemp} implies
\[
(f, v_h-E_hv_h) - b_h(w_h, v_h-E_hv_h)\lesssim \left(\|u-w_h\|_{\varepsilon, h} + \mathrm{osc}_h(f)\right)\|v_h\|_{\varepsilon, h}.
\]
Finally we can finish the proof by combining \eqref{eq:ahtemp} and the last inequality.
\end{proof}

\begin{theorem}\label{thm:ipmwxpriori1}
Let $u\in H_0^2(\Omega)$ be the solution of problem~\eqref{fourthorderpertub}, and $u_h\in V_{h0}$ be the discrete solution of IPMWX method~\eqref{ipdgmwx}. Then we have
\[
\|u-u_h\|_{\varepsilon, h}\lesssim \inf_{w_h\in V_{h0}}\|u-w_h\|_{\varepsilon, h} + \mathrm{osc}_h(f).
\]
\end{theorem}
\begin{proof}
By \eqref{eq:coercivity} and \eqref{ipdgmwx} with $v_h=u_h-w_h$,
\begin{align*}
 \|u_h-w_h\|_{\varepsilon ,h}^2\lesssim &\varepsilon^2 a_h(u_h-w_h, v_h)+b_h(u_h-w_h, v_h) \\
 = & (f, v_h) - \varepsilon^2 a_h(w_h, v_h)-b_h(w_h, v_h).
\end{align*}
Taking $v=E_hv_h$ in \eqref{weak form}, we have
\begin{equation}
  \varepsilon^2 a(u,E_hv_h)+b(u,E_hv_h)=(f,E_hv_h).
\end{equation}
Thus it follows
\begin{align*}
\|u_h-w_h\|_{\varepsilon ,h}^2\lesssim & (f, v_h-E_hv_h)-\varepsilon^2 a_h(w_h, v_h-E_hv_h) - b_h(w_h, v_h-E_hv_h) \\
& + \varepsilon^2 a_h(u-w_h, E_hv_h) + b_h(u-w_h, E_hv_h).
\end{align*}
We get from Cauchy-Schwarz inequality, inverse inequality, and \eqref{eq:Ehestimate3}
\begin{align*}
&\varepsilon^2 a_h(u-w_h, E_hv_h) + b_h(u-w_h, E_hv_h) \\
=& \varepsilon^2 a_h(u-w_h, E_hv_h) + \sum_{K \in \mathcal {T}_h}(\nabla (u-w_h), \nabla E_hv_h)_K  +\sum_{F \in \mathcal{F}_{h}}(\{\partial_{n_F} E_hv_h\} , \llbracket w_h \rrbracket)_F \\
\lesssim &\|u-w_h\|_{\varepsilon, h}\|E_hv_h\|_{\varepsilon, h} \lesssim \|u-w_h\|_{\varepsilon, h}\|v_h\|_{\varepsilon, h}.
\end{align*}
Then we acquire from \eqref{eq:nonconformestimate}
\[
\|u_h-w_h\|_{\varepsilon ,h}^2\lesssim \left(\|u-w_h\|_{\varepsilon, h} + \mathrm{osc}_h(f)\right)\|u_h-w_h\|_{\varepsilon, h},
\]
which indicates
\[
\|u_h-w_h\|_{\varepsilon ,h}\lesssim \|u-w_h\|_{\varepsilon, h} + \mathrm{osc}_h(f).
\]
We can end the proof by using the triangle inequality.
\end{proof}

Let $I_h$ be the Morley-Wang-Xu interpolation operator corresponding to the degrees of freedom~\eqref{eq:mwxdofs}.
We have the following estimate (cf. \cite[Lemma~3]{WangXu2006})
\begin{equation}\label{eq:Ihestimate1}
\|v-I_hv\|_{0, K}+h_K|v-I_hv|_{1, K}+h_K^2|v-I_hv|_{2, K}\lesssim  h_K^{\min\{s,3\}}|v|_{s,K}
\end{equation}
for any $K\in\mathcal{T}_h$ and $v\in H^{s}(\Omega)$ with $s\geq2$.
By a standard scaling argument and the trace inequality, we also have (cf. \cite[(4.4)]{NilssenTaiWinther2001} and \cite[(3.10)]{WangXuHu2006})
\begin{equation}\label{eq:Ihestimate2}
\interleave v-I_hv\interleave^2\lesssim  h|v|_{1}|v|_{2}\quad \forall~ v\in H^2(\Omega).
\end{equation}

Combining Theorem~\ref{thm:ipmwxpriori1} and \eqref{eq:Ihestimate1}, we get the following a priori error estimate.
\begin{corollary}\label{cor:erroripmwx}
Let $u\in H_0^2(\Omega)$ be the solution of problem~\eqref{fourthorderpertub}, and $u_h\in V_{h0}$ be the discrete solution of IPMWX method~\eqref{ipdgmwx}. Assume $u\in H^{s}(\Omega)$ with some $s>2$, then we have
\[
\|u-u_h\|_{\varepsilon, h}\lesssim (\varepsilon+h)h^{\min{\{s,3\}}-2}\|u\|_{s} + \mathrm{osc}_h(f).
\]
\end{corollary}

Next we consider the influence of the boundary layers.
Let $u^0\in H_0^1(\Omega)$ be the solution of the Poisson equation
\begin{equation}\label{poisson}
  \begin{cases}
  -\Delta u^0=f \quad\;\; \textrm{in}~\Omega, \\
  u^0=0 \quad\quad\quad\; \textrm{on}~\partial\Omega.
  \end{cases}
\end{equation}
Assume we have the following regularities
\begin{equation}\label{eq:regularity1}
|u|_2+\varepsilon |u|_3\lesssim \varepsilon^{-1/2}\|f\|_0,
\end{equation}
\begin{equation}\label{eq:regularity2}
|u-u^0|_1\lesssim \varepsilon^{1/2}\|f\|_0,
\end{equation}
\begin{equation}\label{eq:regularity3}
\|u^0\|_2\lesssim \|f\|_0.
\end{equation}
The regularities \eqref{eq:regularity1}-\eqref{eq:regularity2} have been proved in \cite{NilssenTaiWinther2001, GuzmanLeykekhmanNeilan2012}
and \eqref{eq:regularity3} is well-known \cite{Grisvard1985} when the domain $\Omega$ is convex.

\begin{theorem}
Let $u\in H_0^2(\Omega)$ be the solution of problem~\eqref{fourthorderpertub}, and $u_h\in V_{h0}$ be the discrete solution of IPMWX method~\eqref{ipdgmwx}. Assume the regularities \eqref{eq:regularity1}-\eqref{eq:regularity3} hold, then we have
\begin{equation}\label{eq:boundarylayerErrEstimate1}
\|u-u_h\|_{\varepsilon, h}\lesssim h^{1/2}\|f\|_0,
\end{equation}
\begin{equation}\label{eq:boundarylayerErrEstimate2}
\|u^0-u_h\|_{\varepsilon, h}\lesssim (\varepsilon^{1/2}+h^{1/2})\|f\|_0.
\end{equation}
\end{theorem}
\begin{proof}
We adopt the similar argument as in \cite{NilssenTaiWinther2001}.
It follows from \eqref{eq:Ihestimate1} and \eqref{eq:regularity1}
\begin{equation}\label{eq:temp3}
\varepsilon^2|u-I_hu|_{2, h}^2\lesssim h\varepsilon^2|u|_2|u|_3\lesssim h\|f\|_0^2.
\end{equation}
Taking $v=u-u^0$ in \eqref{eq:Ihestimate2}, we get from \eqref{eq:regularity1}-\eqref{eq:regularity3}
\[
\interleave u-u^0-I_h(u-u^0)\interleave^2\lesssim  h|u-u^0|_{1}|u-u^0|_{2}\leq h|u-u^0|_{1}(|u|_{2}+|u^0|_{2})\lesssim h\|f\|_0^2.
\]
By \eqref{eq:Ihestimate1} and \eqref{eq:regularity3}, it holds
\[
\interleave u^0-I_hu^0\interleave\lesssim  h|u^0|_{2}\lesssim h\|f\|_0.
\]
Thus we obtain from the last two inequalities
\[
\interleave u-I_hu\interleave\lesssim h^{1/2}\|f\|_0.
\]
We can achieve \eqref{eq:boundarylayerErrEstimate1} by using Theorem~\ref{thm:ipmwxpriori1} and \eqref{eq:temp3}.

Then according to the triangle inequality and \eqref{eq:regularity1}-\eqref{eq:regularity3}, it follows
\begin{align*}
\|u^0-u_h\|_{\varepsilon, h} &\leq \|u^0-u\|_{\varepsilon, h}+\|u-u_h\|_{\varepsilon, h}\leq \varepsilon|u^0-u|_{2}+|u^0-u|_{1}+\|u-u_h\|_{\varepsilon, h} \\
&\leq \varepsilon|u^0|_{2}+\varepsilon|u|_{2}+|u^0-u|_{1}+\|u-u_h\|_{\varepsilon, h}\lesssim \varepsilon^{1/2}\|f\|_0+\|u-u_h\|_{\varepsilon, h},
\end{align*}
which together with \eqref{eq:boundarylayerErrEstimate1} implies \eqref{eq:boundarylayerErrEstimate2}.
\end{proof}

\begin{remark}\label{remarknew}\rm
In the case of $\varepsilon\lesssim h^{\gamma}$ with $\gamma>1$, a better error estimate of $\|u-u_h\|_{\varepsilon, h}$ can be derived if we apply the Nitsche's technique to the IPMWX method~\eqref{ipdgmwx} as in \cite{GuzmanLeykekhmanNeilan2012}, i.e. impose the boundary condition $u=\partial_nu=0$ weakly.
To be specific, the modified IPMWX method is to find $u_h^N \in V_{h} $ such that
\begin{equation*}
  \varepsilon^2 a_h^N(u_h^N,v_h)+b_h(u_h,v_h)=(f,v_h) \quad \forall~v_h \in V_{h},
\end{equation*}
where
\begin{align*}
a_h^N(u_h^N,v_h):=a_h(u_h^N,v_h) - \sum\limits_{F \in \mathcal{F}_{h}^{\partial}}\Big(&(\partial_{nn}^2u_h^N, \partial_{n}v_h)_F+(\partial_{n}u_h^N, \partial_{nn}^2v_h)_F \\
&-\dfrac{\sigma_1}{h_{F}}(\partial_{n}u_h^N, \partial_{n}v_h)_F-\dfrac{\sigma_2}{h_{F}^3}(u_h^N, v_h)_F\Big),
\end{align*}
with $\mathcal{F}_{h}^{\partial}:=\mathcal{F}_{h}\backslash \mathcal{F}_{h}^{i}$.
Then if $u^0\in H^s(\Omega)$ with $2\leq s\leq 3$, we have (cf. \cite[Theorem 1]{GuzmanLeykekhmanNeilan2012})
\[
\|u-u_h^N\|_{\varepsilon, h}\lesssim \varepsilon^{1/2}\|f\|_0 + h^{s-1}\|u^0\|_s.
\]
This estimate is optimal if $\varepsilon\lesssim h^{2(s-1)}$, even though the exact solution $u$ is not robustly smooth enough with respect to the parameter $\varepsilon$.
And it gives a better convergence rate than $0.5$ if $\varepsilon\lesssim h^{\gamma}$ with $\gamma>1$.
\end{remark}

\section{Super penalty Morley-Wang-Xu element method}

In this section,  we will design another Morley-Wang-Xu element method with penalty for problem~\eqref{fourthorderpertub} by applying the super penalty methods in \cite{BabuvskaZlamal1973, BrezziManziniMariniPietraEtAl2000} to the Laplace operator. The super penalty method possesses rather simple variational formulation. A similar penalty method for the Laplace operator is the weakly over-penalized symmetric interior penalty method developed in \cite{BrennerOwensSung2008, BrennerOwensSung2012}.

Define discrete bilinear form
\begin{align*}
  \widetilde{b}_{h}(w,v):= & \sum_{K \in \mathcal {T}_h}(\nabla w, \nabla v)_K + \sum_{F \in \mathcal{F}_{h}} \frac{1}{h_{F}^{2p+1}}(\llbracket w \rrbracket, \llbracket v \rrbracket)_F \quad \forall~ w, v\in H^{1}(\Omega)+V_{h},
\end{align*}
where $0<p\leq 1$.
The super penalty Morley-Wang-Xu (SPMWX) element method for problem~\eqref{fourthorderpertub}
is to find $\widetilde{u}_h \in V_{h0} $ such that
\begin{equation}\label{pmwx}
  \varepsilon^2 a_h(\widetilde{u}_h,v_h)+\widetilde{b}_h(\widetilde{u}_h,v_h)=(f,v_h) \quad \forall~v_h \in V_{h0}.
\end{equation}

For any $w\in H^1(\Omega)+V_h$, define some discrete norms
\begin{gather*}
  \interleave w \interleave_p^2:=|w|_{1,h}^2+\sum_{F \in \mathcal{F}_{h}} h_F^{-(2p+1)} \| \llbracket w \rrbracket \|_{0,F}^2 ,\quad
  \| w \|_{\varepsilon,p,h}^2:=\varepsilon^2 |w|_{2,h}^2+\interleave w \interleave_p^2 .
\end{gather*}
It is apparent that
 \begin{equation}\label{eq:coercivity1}
 \| w_h \|_{\varepsilon,p,h}^2\lesssim \varepsilon^2 a_h(w_h,w_h)+\widetilde{b}_h(w_h,w_h)\ \ \forall~w_h \in V_{h}.
\end{equation}

By \eqref{eq:Ehestimate3}, we also have for any $w_h \in V_{h0}$
\begin{gather}
 \|E_{h}w_h\|_{\varepsilon,p,h}=\|E_{h}w_h\|_{\varepsilon,h} \lesssim \|w_h\|_{\varepsilon,h} \lesssim \|w_h\|_{\varepsilon,p,h}. \label{eq:Ehestimate4}
\end{gather}

\begin{lemma}
Let $u\in H_0^2(\Omega)$ be the solution of problem~\eqref{fourthorderpertub}. It holds for any $w_h, v_h\in V_{h0}$
\begin{align}
&(f, v_h-E_hv_h)-\varepsilon^2 a_h(w_h, v_h-E_hv_h) - \widetilde{b}_h(w_h, v_h-E_hv_h) \notag\\
\lesssim & \left(\|u-w_h\|_{\varepsilon,p,h} +h^p|u|_1 + \mathrm{osc}_h(f)\right)\|v_h\|_{\varepsilon,p,h}. \label{eq:nonconformestimate1}
\end{align}
\end{lemma}
\begin{proof}
Applying integration by parts, it follows
\begin{align}
&(f, v_h-E_hv_h)-\widetilde{b}_h(w_h, v_h-E_hv_h) \notag \\
=&\sum_{K\in\mathcal{T}_h}(f+\Delta w_h, v_h-E_hv_h)_{K} - \sum_{F \in \mathcal{F}_{h}^i}(\llbracket\partial_{n_F} w_h\rrbracket , \{ v_h-E_hv_h \})_F \notag \\
 &
  -\sum_{F \in \mathcal{F}_{h}}(\{\partial_{n_F} w_h \} , \llbracket v_h-E_hv_h \rrbracket)_F  - \sum_{F \in \mathcal{F}_{h}} \frac{1}{h^{2p+1}_{F}}(\llbracket u-w_h \rrbracket, \llbracket v_h-E_hv_h \rrbracket)_F. \label{eq:ptemp1}
\end{align}
Using the Cauchy-Schwarz inequality, the inverse inequality and the triangle inequality, we have
\begin{align}
-\sum_{F \in \mathcal{F}_{h}}(\{\partial_{n_F} w_h \} , \llbracket v_h-E_hv_h \rrbracket)_F
\leq & \sum_{F \in \mathcal{F}_{h}} h^{p+\frac{1}{2}}_{F}\|\{\partial_{n_F} w_h \}\|_{0,F}
 h^{-p-\frac{1}{2}}_{F}\|\llbracket v_h \rrbracket\|_{0,F}  \notag \\
 \lesssim & h^p|w_h|_{1,h}\|v_h\|_{\varepsilon,p,h} \notag\\
 \leq & h^p(|u|_{1}+|u-w_h|_{1,h})\|v_h\|_{\varepsilon,p,h} \notag \\
 \leq & h^p(|u|_{1}+\|u-w_h\|_{\varepsilon,p,h})\|v_h\|_{\varepsilon,p,h}, \label{eq:ptemp2}
\end{align}
\begin{align}
-\sum_{F \in \mathcal{F}_{h}} \frac{1}{h^{2p+1}_{F}}(\llbracket u-w_h \rrbracket, \llbracket v_h-E_hv_h \rrbracket)_F = & -\sum_{F \in \mathcal{F}_{h}} \frac{1}{h^{2p+1}_{F}}(\llbracket u-w_h \rrbracket, \llbracket v_h\rrbracket)_F \notag \\
\lesssim & \|u-w_h\|_{\varepsilon,p,h}\|v_h\|_{\varepsilon,p,h}. \label{eq:ptemp3}
\end{align}
Combining \eqref{eq:fdeltawhtemp}-\eqref{eq:partialnwhtemp} and \eqref{eq:ptemp1}-\eqref{eq:ptemp3}, we get
\[
(f, v_h-E_hv_h)- \widetilde{b}_h(w_h, v_h-E_hv_h)
\lesssim \left(\|u-w_h\|_{\varepsilon,p,h} +h^p|u|_1 + \mathrm{osc}_h(f)\right)\|v_h\|_{\varepsilon,p,h},
\]
which together with \eqref{eq:ahtemp} implies \eqref{eq:nonconformestimate1}.
\end{proof}

\begin{theorem}
Let $u\in H_0^2(\Omega)$ be the solution of problem~\eqref{fourthorderpertub}, and $\widetilde{u}_h\in V_{h0}$ be the discrete solution of the SPMWX method~\eqref{pmwx}. Then we have
\[
\|u-\widetilde{u}_h\|_{\varepsilon,p,h}\lesssim \inf_{w_h\in V_{h0}}\|u-w_h\|_{\varepsilon,p,h} + h^p|u|_{1}+ \mathrm{osc}_h(f).
\]
\end{theorem}
\begin{proof}
By ~\eqref{eq:coercivity1} and ~\eqref{pmwx} with $v_h=\widetilde{u}_h-w_h$, and following the line of the proof of Theorem \ref{thm:ipmwxpriori1}, we have
\begin{align*}
\|\widetilde{u}_h-w_h\|_{\varepsilon,p,h}^2\lesssim & (f, v_h-E_hv_h)-\varepsilon^2 a_h(w_h, v_h-E_hv_h) - \widetilde{b}_h(w_h, v_h-E_hv_h) \\
& + \varepsilon^2 a_h(u-w_h, E_hv_h) + \widetilde{b}_h(u-w_h, E_hv_h).
\end{align*}
Combined with Cauchy-Schwarz inequality, inverse inequality, and \eqref{eq:Ehestimate4}, it holds
\begin{align*}
&\varepsilon^2 a_h(u-w_h, E_hv_h) + \widetilde{b}_h(u-w_h, E_hv_h) \\
=& \varepsilon^2 a_h(u-w_h, E_hv_h) + \sum_{K \in \mathcal {T}_h}(\nabla (u-w_h), \nabla E_hv_h)_K  \\
\lesssim &\|u-w_h\|_{\varepsilon,p,h}\|E_hv_h\|_{\varepsilon,p,h} \lesssim \|u-w_h\|_{\varepsilon,p,h}\|v_h\|_{\varepsilon,p,h}.
\end{align*}
Then from \eqref{eq:nonconformestimate1}, we have
\begin{align*}
\|\widetilde{u}_h-w_h\|_{\varepsilon,p,h}\lesssim & \|u-w_h\|_{\varepsilon,p,h} +h^p|u|_{1} + \mathrm{osc}_h(f).
\end{align*}
Using the triangle inequality, we have
\begin{align*}
\|u-\widetilde{u}_h\|_{\varepsilon,p,h}\lesssim & \inf_{w_h\in V_{h0}}\|u-w_h\|_{\varepsilon,p,h} + h^p|u|_{1}+ \mathrm{osc}_h(f),
\end{align*}
as required.
\end{proof}

\begin{corollary}\label{cor:errorspmwx}
Let $u\in H_0^2(\Omega)$ be the solution of problem~\eqref{fourthorderpertub}, and $\widetilde{u}_h\in V_{h0}$ be the discrete solution of the SPMWX method~\eqref{pmwx}. Assume $u\in H^{s}(\Omega)$ with some $s>2$, then we have
\[
\|u-\widetilde{u}_h\|_{\varepsilon,p,h}\lesssim \left(\varepsilon h^{\min{\{s,3\}}-2} + h^{\min{\{s,2p+1\}}-p-1}\right)\|u\|_{s} + \mathrm{osc}_h(f).
\]
In particular, if $p=1$, then
\[
\|u-\widetilde{u}_h\|_{\varepsilon,p,h}\lesssim h^{\min{\{s,3\}}-2}\|u\|_{s} + \mathrm{osc}_h(f).
\]
If $p=(s-1)/2$ with $s\leq 3$, then
\[
\|u-\widetilde{u}_h\|_{\varepsilon,p,h}\lesssim\left(\varepsilon + h^{(3-s)/2}\right)h^{s-2}\|u\|_{s} + \mathrm{osc}_h(f).
\]\end{corollary}

\section{Numerical Results}

In this section, we will report some numerical results for the IPMWX method~\eqref{ipdgmwx} and the SPMWX method~\eqref{pmwx}.
Set $\sigma=5$ in the IPMWX method~\eqref{ipdgmwx}.

\subsection{Example 1}
In the first example, let $\Omega$ be the unit square $(0, 1)^2$, and we use the uniform triangulation $\mathcal T_h$ of $\Omega$.
The exact solution is taken as
\[
u(x_1, x_2)=(\sin(\pi x_1)\sin(\pi x_2))^2.
\]
The right hand side $f$ is computed from~\eqref{fourthorderpertub}.

The errors $\|u-u_h\|_{\varepsilon,h}$ for different $h$ and $\varepsilon$
are shown in Table~\ref{tab:eneryerrorIPMWX} for the IPMWX method~\eqref{ipdgmwx}. We can see that $\|u-u_h\|_{\varepsilon,h}=O(h)$ when $\varepsilon=O(1)$, and $\|u-u_h\|_{\varepsilon, h}=O(h^2)$ when $\varepsilon\ll1$, which coincide with Corollary~\ref{cor:erroripmwx}.
The optimal convergence rate of $\|u-u_h\|_{\varepsilon, h}$ when $\varepsilon\ll1$ for this example is achieved since the exact solution $u$ is robustly smooth enough with respect to the parameter $\varepsilon$. However, such a robust high regularity of $u$ does not hold for a general right hand side $f$, as indicated by \eqref{eq:regularity1}-\eqref{eq:regularity2}. We will observe the boundary layers in the next example.
Then set $p=1$ in the SPMWX method~\eqref{pmwx}.
The errors $\|u-\widetilde{u}_h\|_{\varepsilon,p,h}$ for different $h$ and $\varepsilon$
are shown in Table~\ref{tab:eneryerrorSPMWX} for the SPMWX method~\eqref{pmwx}. It can be observed that
$\|u-\widetilde{u}_h\|_{\varepsilon,p,h}=O(h)$ no matter what value the parameter $\varepsilon$ takes, which agrees with Corollary~\ref{cor:errorspmwx}.

\begin{table}[htbp]
  \centering
  \caption{Error $\|u-u_h\|_{\varepsilon, h}$ of the IPMWX method for Example 1}
\resizebox{\textwidth}{!}{ %
\begin{tabular}{|c|c|c|c|c|c|c|c|}
\hline
$\varepsilon$ & $h$   & $2^{-2}$ & $2^{-3}$ & $2^{-4}$ & $2^{-5}$ & $2^{-6}$ & $2^{-7}$ \bigstrut\\
\hline
\multirow{2}[4]{*}{$1$} & $\|u-u_h\|_{\varepsilon, h}$ & 1.053E+01 & 5.938E+00 & 3.076E+00 & 1.553E+00 & 7.781E-01 & 3.893E-01 \bigstrut\\
\cline{2-8}      & rate & -     & 0.83  & 0.95  & 0.99  & 1.00  & 1.00  \bigstrut\\
\hline
\multirow{2}[4]{*}{$10^{-1}$} & $\|u-u_h\|_{\varepsilon, h}$ & 8.613E-01 & 5.004E-01 & 2.835E-01 & 1.512E-01 & 7.726E-02 & 3.886E-02 \bigstrut\\
\cline{2-8}      & rate & -     & 0.78  & 0.82  & 0.91  & 0.97  & 0.99  \bigstrut\\
\hline
\multirow{2}[4]{*}{$10^{-2}$} & $\|u-u_h\|_{\varepsilon, h}$ & 3.650E-01 & 1.046E-01 & 2.929E-02 & 1.405E-02 & 7.020E-03 & 3.632E-03 \bigstrut\\
\cline{2-8}      & rate & -     & 1.80  & 1.84  & 1.06  & 1.00  & 0.95  \bigstrut\\
\hline
\multirow{2}[4]{*}{$10^{-3}$} & $\|u-u_h\|_{\varepsilon, h}$ & 3.796E-01 & 1.545E-01 & 3.832E-02 & 8.846E-03 & 1.812E-03 & 3.992E-04 \bigstrut\\
\cline{2-8}      & rate & -     & 1.30  & 2.01  & 2.11  & 2.29  & 2.18  \bigstrut\\
\hline
\multirow{2}[4]{*}{$10^{-4}$} & $\|u-u_h\|_{\varepsilon, h}$ & 3.798E-01 & 1.555E-01 & 3.915E-02 & 9.585E-03 & 2.367E-03 & 5.832E-04 \bigstrut\\
\cline{2-8}      & rate & -     & 1.29  & 1.99  & 2.03  & 2.02  & 2.02  \bigstrut\\
\hline
\multirow{2}[4]{*}{$10^{-5}$} & $\|u-u_h\|_{\varepsilon, h}$ & 3.798E-01 & 1.555E-01 & 3.916E-02 & 9.593E-03 & 2.375E-03 & 5.910E-04 \bigstrut\\
\cline{2-8}      & rate & -     & 1.29  & 1.99  & 2.03  & 2.01  & 2.01  \bigstrut\\
\hline
\multirow{2}[4]{*}{$0$} & $\|u-u_h\|_{\varepsilon, h}$ & 3.798E-01 & 1.555E-01 & 3.916E-02 & 9.593E-03 & 2.375E-03 & 5.911E-04 \bigstrut\\
\cline{2-8}      & rate & -     & 1.29  & 1.99  & 2.03  & 2.01  & 2.01  \bigstrut\\
\hline
\end{tabular}%
    }
  \label{tab:eneryerrorIPMWX}%
\end{table}%

\begin{table}[htbp]
  \centering
  \caption{Error $\|u-\widetilde{u}_h\|_{\varepsilon,p,h}$ of the SPMWX method for Example 1}
\resizebox{\textwidth}{!}{ %
    \begin{tabular}{|c|c|c|c|c|c|c|c|}
    \hline
    $\varepsilon$ & $h$   & $2^{-2}$ & $2^{-3}$ & $2^{-4}$ & $2^{-5}$ & $2^{-6}$ & $2^{-7}$ \bigstrut\\
    \hline
    \multirow{2}[4]{*}{$1$} & $\|u-\widetilde{u}_h\|_{\varepsilon,p,h}$ & 1.071E+01 & 5.938E+00 & 3.060E+00 & 1.542E+00 & 7.726E-01 & 3.865E-01 \bigstrut\\
\cline{2-8}          & rate & -     & 0.85  & 0.96  & 0.99  & 1.00  & 1.00  \bigstrut\\
    \hline
    \multirow{2}[4]{*}{$10^{-1}$} & $\|u-\widetilde{u}_h\|_{\varepsilon,p,h}$ & 1.254E+00 & 6.670E-01 & 3.372E-01 & 1.690E-01 & 8.457E-02 & 4.229E-02 \bigstrut\\
\cline{2-8}          & rate & -     & 0.91  & 0.98  & 1.00  & 1.00  & 1.00  \bigstrut\\
    \hline
    \multirow{2}[4]{*}{$10^{-2}$} & $\|u-\widetilde{u}_h\|_{\varepsilon,p,h}$ & 8.198E-01 & 3.848E-01 & 1.918E-01 & 9.600E-02 & 4.802E-02 & 2.401E-02 \bigstrut\\
\cline{2-8}          & rate & -     & 1.09  & 1.00  & 1.00  & 1.00  & 1.00  \bigstrut\\
    \hline
    \multirow{2}[4]{*}{$10^{-3}$} & $\|u-\widetilde{u}_h\|_{\varepsilon,p,h}$ & 8.143E-01 & 3.807E-01 & 1.898E-01 & 9.501E-02 & 4.753E-02 & 2.377E-02 \bigstrut\\
\cline{2-8}          & rate & -     & 1.10  & 1.00  & 1.00  & 1.00  & 1.00  \bigstrut\\
    \hline
    \multirow{2}[4]{*}{$10^{-4}$} & $\|u-\widetilde{u}_h\|_{\varepsilon,p,h}$ & 8.142E-01 & 3.807E-01 & 1.897E-01 & 9.500E-02 & 4.752E-02 & 2.376E-02 \bigstrut\\
\cline{2-8}          & rate & -     & 1.10  & 1.00  & 1.00  & 1.00  & 1.00  \bigstrut\\
    \hline
    \multirow{2}[4]{*}{$10^{-5}$} & $\|u-\widetilde{u}_h\|_{\varepsilon,p,h}$ & 8.142E-01 & 3.807E-01 & 1.897E-01 & 9.500E-02 & 4.752E-02 & 2.376E-02 \bigstrut\\
\cline{2-8}          & rate & -     & 1.10  & 1.00  & 1.00  & 1.00  & 1.00  \bigstrut\\
    \hline
    \multirow{2}[4]{*}{$0$} & $\|u-\widetilde{u}_h\|_{\varepsilon,p,h}$ & 8.142E-01 & 3.807E-01 & 1.897E-01 & 9.500E-02 & 4.752E-02 & 2.376E-02 \bigstrut\\
\cline{2-8}          & rate & -     & 1.10  & 1.00  & 1.00  & 1.00  & 1.00  \bigstrut\\
    \hline
    \end{tabular}%
    }
  \label{tab:eneryerrorSPMWX}%
\end{table}%

\subsection{Example 2}

Now we examine the performance of the IPMWX method~\eqref{ipdgmwx} for problem~\eqref{fourthorderpertub} with boundary layers on the same domain $\Omega$ and uniform triangulation $\mathcal T_h$ as in Example 1.
Take the exact solution of the Poisson equation~\eqref{poisson} to be
\[
u^0(x_1, x_2)=\sin(\pi x_1)\sin(\pi x_2).
\]
Then the right hand term for problems~\eqref{fourthorderpertub} and \eqref{poisson} is set to be
\[
f(x_1, x_2)=-\Delta u^0=2\pi^2\sin(\pi x_1)\sin(\pi x_2).
\]
The solution $u$ for problem~\eqref{fourthorderpertub} with this right hand term possesses strong boundary layers when $\varepsilon$ is very small.
The explicit expression of $u$ is unknown.
Here we take $\varepsilon=10^{-6}$.
Numerical errors $\|u^0-u_h\|_0$, $|u^0-u_h|_{1,h}$ and $\|u^0-u_h\|_{\varepsilon, h}$ are shown in Table~\ref{tab:errorsIPMWXex2}.
It can be observed that the convergence rate of $\|u^0-u_h\|_{\varepsilon, h}$ is just $0.5$, as indicated by \eqref{eq:boundarylayerErrEstimate2}.
We can recover the optimal convergence rates for errors $\|u^0-u_h\|_0$, $|u^0-u_h|_{1,h}$ and $\|u^0-u_h\|_{\varepsilon, h}$ by using the Nitsche's
technique as in \cite{GuzmanLeykekhmanNeilan2012} (see Remark~\ref{remarknew}).

\begin{table}
  \centering
  \caption{Numerical errors of the IPMWX method for Example 2 with $\varepsilon=10^{-6}$}\label{tab:errorsIPMWXex2}
  \begin{tabular}{|c|c|c|c|c|c|c|}
    \hline
    $h$ & $\|u^0-u_h\|_0$ & rate & $|u^0-u_h|_{1,h}$ & rate & $\|u^0-u_h\|_{\varepsilon, h}$ & rate \\
  \hline  $2^{-2}$ & 5.2098E-02 & - & 1.5242E+00 & - & 1.6884E+00 & - \\
  \hline  $2^{-3}$ & 1.5193E-02 & 1.78 & 9.3654E-01 & 0.70 & 1.0357E+00 & 0.71 \\
  \hline  $2^{-4}$ & 5.3450E-03 & 1.51 & 6.3412E-01 & 0.56 & 7.0261E-01 & 0.56 \\
  \hline  $2^{-5}$ & 2.0058E-03 & 1.41 & 4.4401E-01 & 0.51 & 4.9227E-01 & 0.51 \\
  \hline  $2^{-6}$ & 7.8914E-04 & 1.35 & 3.1327E-01 & 0.50 & 3.4738E-01 & 0.50 \\
  \hline  $2^{-7}$ & 3.2819E-04 & 1.27 & 2.2140E-01 & 0.50 & 2.4552E-01 & 0.50 \\
  \hline  $2^{-8}$ & 1.4462E-04 & 1.18 & 1.5653E-01 & 0.50 & 1.7359E-01 & 0.50 \\
    \hline
  \end{tabular}
\end{table}

\subsection{Example 3}

At last, we testify the IPMWX method~\eqref{ipdgmwx} in three dimensions.
Let $\Omega$ be the unit cube $(0, 1)^3$, and take the uniform triangulation.
Set
\[
u(x_1, x_2, x_3)=(\sin(\pi x_1)\sin(\pi x_2)\sin(\pi x_3))^2,
\]
and the right hand side $f$ is computed from~\eqref{fourthorderpertub}.

The errors $\|u-u_h\|_{\varepsilon,h}$ for different $h$ and $\varepsilon$
are shown in Table~\ref{tab:eneryerrorIPMWXex3} for the IPMWX method~\eqref{ipdgmwx}.
As indicated by Corollary~\ref{cor:erroripmwx},
it is numerically shown again that $\|u-u_h\|_{\varepsilon,h}=O(h)$ when $\varepsilon=O(1)$, and $\|u-u_h\|_{\varepsilon, h}=O(h^2)$ when $\varepsilon\ll1$.

\begin{table}[htbp]
  \centering
  \caption{Error $\|u-u_h\|_{\varepsilon, h}$ of the IPMWX method for Example 3}
\resizebox{\textwidth}{!}{ %
    \begin{tabular}{|c|c|c|c|c|c|c|}
    \hline
    $\varepsilon$ & $h$   & $2^{-1}$ & $2^{-2}$ & $2^{-3}$ & $2^{-4}$ & $2^{-5}$ \bigstrut\\
    \hline
    \multirow{2}[4]{*}{$1$} & $\|u-u_h\|_{\varepsilon, h}$ & 1.189E+01 & 7.346E+00 & 3.635E+00 & 1.798E+00 & 8.957E-01 \bigstrut\\
\cline{2-7}          & order & -     & 0.69  & 1.02  & 1.02  & 1.01  \bigstrut\\
    \hline
    \multirow{2}[4]{*}{$10^{-1}$} & $\|u-u_h\|_{\varepsilon, h}$ & 1.212E+00 & 6.297E-01 & 3.379E-01 & 1.762E-01 & 8.913E-02 \bigstrut\\
\cline{2-7}          & order & -     & 0.94  & 0.90  & 0.94  & 0.98  \bigstrut\\
    \hline
    \multirow{2}[4]{*}{$10^{-2}$} & $\|u-u_h\|_{\varepsilon, h}$ & 7.835E-01 & 2.634E-01 & 7.093E-02 & 2.214E-02 & 9.075E-03 \bigstrut\\
\cline{2-7}          & order & -     & 1.57  & 1.89  & 1.68  & 1.29  \bigstrut\\
    \hline
    \multirow{2}[4]{*}{$10^{-3}$} & $\|u-u_h\|_{\varepsilon, h}$ & 7.787E-01 & 2.593E-01 & 6.668E-02 & 1.698E-02 & 4.280E-03 \bigstrut\\
\cline{2-7}          & order & -     & 1.59  & 1.96  & 1.97  & 1.99  \bigstrut\\
    \hline
    \multirow{2}[4]{*}{$10^{-4}$} & $\|u-u_h\|_{\varepsilon, h}$ & 7.786E-01 & 2.593E-01 & 6.667E-02 & 1.697E-02 & 4.269E-03 \bigstrut\\
\cline{2-7}          & order & -     & 1.59  & 1.96  & 1.97  & 1.99  \bigstrut\\
    \hline
    \multirow{2}[4]{*}{$10^{-5}$} & $\|u-u_h\|_{\varepsilon, h}$ & 7.786E-01 & 2.593E-01 & 6.667E-02 & 1.697E-02 & 4.269E-03 \bigstrut\\
\cline{2-7}          & order & -     & 1.59  & 1.96  & 1.97  & 1.99  \bigstrut\\
    \hline
    \multirow{2}[4]{*}{$0$} & $\|u-u_h\|_{\varepsilon, h}$ & 7.786E-01 & 2.593E-01 & 6.667E-02 & 1.697E-02 & 4.269E-03 \bigstrut\\
\cline{2-7}          & order & -     & 1.59  & 1.96  & 1.97  & 1.99  \bigstrut\\
    \hline
    \end{tabular}%
    }
  \label{tab:eneryerrorIPMWXex3}%
\end{table}%


\end{document}